\documentclass[12pt]{article}

\usepackage[utf8]{inputenc}
\usepackage{graphicx}
\usepackage[all]{xy}
\usepackage{amssymb}
\usepackage{amsmath} 
\usepackage{amsthm}
\usepackage{verbatim}
\usepackage{latexsym}
\usepackage{mathrsfs}

\newcommand{\rusi}{W}
\newcommand{\rusch}{Q}
\newcommand{\ruszh}{R}
\newenvironment{myItemize}{ 
  \begin{list}{\raisebox{2.2pt}{$\centerdot$}}{%
      \setlength\topsep{5pt}
      \setlength\leftmargin{25pt}
      \setlength\labelwidth{20pt}
      \setlength\itemsep{-2pt}
    }
  }{
  \end{list}
}

\newenvironment{myEnumerate}{ 
  \begin{list}{\labelenumi}{%
      \usecounter{enumi}
      \setlength\leftmargin{25pt}
      \setlength\labelwidth{20pt}
    }
  }{
  \end{list}
}

\newenvironment{myAlphanumerate}{ 
  \begin{list}{\labelenumii}{%
      \usecounter{enumii}
      \setlength\leftmargin{25pt}
      \setlength\labelwidth{20pt}
    }
  }{
  \end{list}
}

\makeindex

\renewcommand{\Cup}{\bigcup}
\renewcommand{\Cap}{\bigcap}

\renewcommand{\a}{\alpha}
\renewcommand{\b}{\beta}
\newcommand{\g}{\gamma}

\renewcommand{\k}{\kappa}
\renewcommand{\l}{\lambda}
\newcommand{\f}{\varphi}
\renewcommand{\o}{\omega}
\newcommand{\s}{\sigma}

\newcommand{\es}{\varnothing}

\newcommand{\A}{\mathcal{A}}

\newcommand{\M}{\mathcal{M}}
\renewcommand{\ll}{\mathscr{L}}

\newcommand{\Q}{\mathbb{Q}} 

\renewcommand{\P}{\mathbb{P}}             

\newcommand{\la}{\langle} 
\newcommand{\ra}{\rangle}
\newcommand{\Po}{\mathcal{P}}             

\newcommand{\ISO}{\operatorname{ISO}}
\newcommand{\Str}{\operatorname{Str}}
\newcommand{\dlo}{{\operatorname{DLO}}}
\newcommand{\clo}{{\operatorname{CLO}}}
\newcommand{\cf}{\operatorname{cf}} 

\newcommand{\hgt}{\operatorname{ht}} 
\newcommand{\CT}{\operatorname{CT}} 
\newcommand{\rest}{\!\restriction\!}
\newcommand{\restl}{\restriction}  
\newcommand{\dom}{\operatorname{dom}}
\newcommand{\ran}{\operatorname{ran}}
 \renewcommand{\le}{\leqslant}  
 \renewcommand{\ge}{\geqslant}  
\newcommand{\sd}{\,\triangle\,}             

 \newcommand{\Sii}{{\Sigma_1^1}}
 
 \newcommand{\Dii}{{\Delta_1^1}}
 \newcommand{\I}{\mathcal{I}}
 \newcommand{\J}{\mathcal{J}}
 \newcommand{\EF}{\operatorname{EF}} 

 \newcommand{\PlOne}{\,{\textrm{\bf I}}}
 \newcommand{\PlTwo}{\textrm{\bf I\hspace{-1pt}I}}

\newcommand{\Sk}{{\operatorname{Sk}}} 
\newcommand{\ZF}{\operatorname{ZF}}
\newcommand{\ZFC}{\operatorname{ZFC}}  
\newcommand{\Borel}{\operatorname{Borel}}  

\swapnumbers
\newtheorem*{Thm*}{Theorem}
\newtheorem{Thm}{Theorem}
\newtheorem{Lemma}[Thm]{Lemma}
\newtheorem{Cor}[Thm]{Corollary}

\newenvironment{claim}[1]{\text{ }\vspace{7pt}\newline\noindent\textbf{Claim #1.}}{\hspace*{\fill}}
\newenvironment{claim*}{\vspace{7pt}\noindent\textbf{Claim.}}{}

\theoremstyle{definition}
\newtheorem{Def}[Thm]{Definition}

\newtheorem{Fact}[Thm]{Fact}

\theoremstyle{remark}
\newtheorem*{Remark}{Remark}
\newtheorem{RemarkN}[Thm]{Remark}

\newcommand{\proofvpara}{\text{}}

\newenvironment{proofVOf}[1] {\hfill\vspace{5pt}\\\noindent \textbf{Proof of #1.}\ignorespaces\renewcommand{\proofvpara}{\text{#1}}}
{\nopagebreak\hspace*{\fill}\mbox{$\square_{\,\proofvpara}$}\\\vspace{-8pt}}

\newenvironment{proofV}[1] {\hfill\\\noindent \textbf{Proof.}\ignorespaces \renewcommand{\proofvpara}{\text{#1}}}
{\nopagebreak\hspace*{\fill}\mbox{$\square_{\,\proofvpara}$}\\}

\author{Tapani Hyttinen and Vadim Kulikov}
\title{On $\Sii$-complete Equivalence Relations on the Generalized Baire Space}

\usepackage[text={16cm,22cm},centering]{geometry}

\begin{document}

\maketitle

\begin{abstract}
  Working with uncountable structures of fixed cardinality,
  we investigate the complexity of certain equivalence relations and show that
  if $V=L$, then many of them are $\Sii$-complete, in particular the
  isomorphism relation of dense linear orders.

  Then we show that it is undecidable in ZFC whether or not the isomorphism relation of 
  a certain well behaved theory (stable, NDOP, NOTOP) is $\Sii$-complete (it is, if $V=L$,
  but can be forced not to be).
\end{abstract}

\vspace{5pt}

Key words: descriptive complexity, generalized Baire space, stability theory.

\vspace{10pt}

2012 MSC: 03C55, 03E47.

\vspace{10pt}

\section*{Introduction}

The descriptive set theory of the generalized Baire space $\k^\k$ for uncountable $\k$
has been initiated in the 1990's, see for example \cite{MV,Ha}, and developed 
further e.g. in~\cite{FHK}. The theory differs from the classical 
case $\k=\o$ in many respects, but similarly as in classical case there is a strong connection
to model theory.

Let $T$ be a complete countable first-order theory, $\M(T)$ the set of models of $T$
with domain $\k$ and $\ISO(T)$ the isomorphism relation on $\M(T)$. In a standard way $\M(T)$
can be viewed as a Borel subset of $2^\k$.
It was established in \cite{FHK}, that in many cases the descriptive complexity of
$\ISO(T)$
is high if and only if $T$ is ``hard'' in terms of the classification theory developed by Shelah~\cite{Sh}.
For example if the isomorphism can be decided with a relatively short Ehrenfeucht-Fra\"iss\'e-game,
then the isomorphism relation is Borel* (Definition \ref{def:Eka}). This result is obtained
by translating between the EF-game and the Borel*-game which are similar in nature.
On the other hand, if the theory is unclassifiable, then
the equivalence relation on $2^\k$ modulo a certain version of the 
non-stationary ideal can be embedded into its isomorphism
relation. A more robust example are the following two theorems:

\begin{Thm*}[\cite{FHK}]\label{thm:ShallowBorell}
  Assume that $\k^{<\k}=\k>\o$ is not weakly inaccessible and $T$ a complete countable 
  first-order theory. 
  If the isomorphism relation $\cong^\k_T$ is Borel, then $T$ is classifiable 
  (superstable, NDOP and NOTOP) and shallow.
  Conversely, if $\k>2^\o$, then if $T$ is classifiable and shallow, then
  $\cong^\k_T$ is Borel.
\end{Thm*}

\begin{Thm*}[\cite{FHK}]
  Suppose $\k=\l^+=2^\l>2^\o$ where $\l^{<\l}=\l$. Let
  $T$ be a first-order theory. Then $T$ is classifiable if and only if for all regular $\mu<\k$,
  $E^\k_\mu\not\le_B\,\cong_T$, where $E^\k_\mu$ is the equivalence on $2^\k$
  modulo the ideal of not $\mu$-stationary sets.
\end{Thm*}

Thus, the vague thesis of \cite{FHK} is that the more complex the theory is according to 
classification theory, the more complex is its isomorphism relation in terms of 
the generalized descriptive set theory at some fixed cardinal $\k$. In this paper we show
that if $V=L$, then there is a counter example to this thesis: 
the theory $T_{\o+\o}$ (see Definition~\ref{def:Too}) is stable with no DOP nor OTOP, its
isomorphism relation can be decided by an EF-game of relatively short length and
its isomorphism relation is $\Sii$-complete (being Borel* at the same time). 
In order to do that, we investigate also other $\Sii$-complete equivalence relations on $\k^\k$
for $\k^{<\k}=\k>\o$ and meanwhile show that the isomorphism relation of dense linear orderings
is $\Sii$-complete, if $V=L$ (without $V=L$ we still get that $\ISO(\k,\dlo)$ is $S_\k$-complete).

Then we show also that the same cannot be proven in ZFC, i.e. in a certain forcing extension $T_{\o+\o}$ is
not even $S_\k$-complete, Corollary~\ref{cor:StableNotCom}.

\paragraph{Acknowledgment.} We wish to thank Sy-David Friedman for the useful discussions we had during the preparation of this paper.
The research was partially supported by the Academy of Finland through its grant WBS 1251557 and the second author
was funded by the Science Foundation of the University of Helsinki.

\section{Some $\Sii$-complete Equivalence Relations in $L$}

In this section we give definitions and show that if $V=L$, then many equivalence relations,
such as the equivalence on $\l^\k$ modulo the non-stationary ideal and the isomorphism relation 
of dense linear orders, are $\Sii$-complete.

\begin{Def}
  We fix an uncountable cardinal $\k$ with the property $|\k^{<\k}|=|\Cup_{\a<\k}\k^\a|=\k$. 
  We use the notation $\a^{<\b}$ to denote both the set of functions from the initial segments of $\b$ to $\a$
  and the cardinality $|\a^{<\b}|$, $\a,\b$ ordinals.
  Our basic space is $\k^\k$, all functions from $\k$ to $\k$, with the topology generated by 
  $$N_p=\{\eta\in \k^\k\mid \eta\supset p\}\quad p\in \k^{<\k}.$$
  This is the generalized Baire space.
  Often we deal with the closed subspaces of $\k^\k$ such as $2^\k$ and $\l^\k$ with $\l<\k$ an infinite
  cardinal. Then the topology on them is the relative subspace topology.
  We fix a one-to-one coding between the models of a fixed countable vocabulary with the universe $\k$ and 
  elements of $2^\k\subset \k^\k$:
  $$\eta\in 2^\k\iff \A_\eta\text{ is a model with }\dom A_\eta=\k.$$
  More precisely, let $\ll$ be a countable relational vocabulary, $\ll=\{R_n\mid n<\o\}$
  and let $\#R_n$ be the arity of $R_n$. Let $\pi\colon \Cup_{n<\o}\{n\}\times\k^{\#R_n}\to\k$
  be a bijection. Given a function $\eta\in 2^\k$, let
  $\A_\eta$ be the structure such that $\dom \A_\eta=\k$ and 
  $$\A_\eta\models R_n(\a_1,\ldots,\a_{\#R_n})\iff \eta(\pi(n,\a_1,\ldots,\a_{\#R_n}))=1.$$
  This is clearly bijective and in some sense continuous -- the further $\eta$ is known the
  larger segment of the model is determined and vice versa.

  The collection of \emph{Borel} sets is the smallest collection of subsets of $\k^\k$ such that: 
  \begin{myItemize}
  \item closed sets are Borel,
  \item if $(A_i)_{i<\k}$ is a sequence of Borel sets, then $\Cup_{i<\k}A_i$, $\Cap_{i<\k}A_i$ and $\k^\k\setminus A_0$
    are Borel.
  \end{myItemize}
  A function $X\to Y$, $X,Y\subset\k^\k$, is \emph{Borel}, if the inverse image of every open set is Borel.

  An equivalence relation $E$ on $X\subset \k^\k$ is \emph{Borel reducible} to an equivalence relation 
  $E'$ on $Y\subset\k^\k$, if there is a Borel function
  $f\colon X\to Y$ such that $\eta E\xi\iff f(\eta)E' f(\xi)$.

  The coding of models to elements of $2^\k$
  can be extended to $\l^\k$ ($\l>2$) via the continuous 
  surjection $\eta\mapsto \xi$, $\xi(\a)=0\iff\eta(\a)=0$, for $\eta\in\l^\k$ and
  $\xi\in 2^\k$.
\end{Def}

A set $A\subset\k^\k$ is $\Sii$, if it is the projection of a closed or Borel set $C\subset \k^\k\times\k^\k$.
It is~$\Dii$, if both $A$ and its complement are~$\Sii$.

The following definition of $\Borel^{*}(\k)$ sets
is from \cite{Bl} in the case $\k =\o$
and from \cite{MV} in the case $\k$ is uncountable.

\begin{Def}\label{def:Eka}
  Let $\a\le\k$ be an ordinal and $\l\le\k$ a cardinal.
  \begin{myEnumerate}
  \item We say that a tree $t$ is a $\k^{+},\a$-tree
    if does not contain chains of order-type $\a$ and
    every element has at most $\k$ successors.
  \item We say that a pair $(t,f)$ is a $\Borel^{*}_{\l}$-code
    if $t$ is a closed $\k^{+},\l$-tree and $f$ is a function
    with domain $t$ such that if $x\in t$ is a leaf, then
    $f(x)$ is a basic open set and otherwise
    $f(x)\in\{\cup ,\cap\}$.
  \item For an element $\eta\in \k^\k$ and
    a $\Borel^{*}_{\l}(\k)$-code $(t,f)$, the $\Borel^{*}$-game
    $B^{*}(\eta ,(t,f))$ is played as follows.
    There are two players, $\PlOne$ and $\PlTwo$. The game
    starts from the root of $t$. At each move,
    if the game is at node $x\in t$ and $f(x)=\cap$,
    then $\PlOne$ chooses an immediate successor $y$ of $x$
    and the game continues from this $y$. If $f(x)=\cup$,
    then $\PlTwo$ makes the choice.
    At limits the game continues from the (unique)
    supremum of the previous moves.
    Finally, if $f(x)$ is a basic open set,
    then the game ends, and $\PlTwo$ wins if $\eta\in f(x)$.
  \item We say that $X\subseteq \k^\k$ is a $\Borel^{*}_{\l}(\k)$ set
    if it has a $\Borel^{*}_{\l}(\k)$-code
    i.e. that there is a $\Borel^{*}_{\l}(\k)$-code
    $(t,f)$ such that for all $\eta\in \k^\k$,
    $\eta\in X$ iff $\PlTwo$ has a winning strategy in the game
    $B^{*}(\eta ,(t,f))$.
  \item In this paper we have fixed an uncountable cardinal $\k$ and
    we will drop $\k$ from the notation, i.e. $\Borel^*=\Borel^*(\k)$
    and we write $\Borel^{*}$ also for the family of all $\Borel^{*}$ sets.
  \end{myEnumerate}
\end{Def}

\begin{Def}
  Given a class $M$ of structures with domain $\k$, let $C(M)\subset 2^\k$ be the set of codes of elements of $M$.
  If $M$ is closed under isomorphism, denote by $\ISO(M)$ the isomorphism relation on $C(M)$.
  If $M=\Str^\k(T)=\{\A\mid \dom\A=\k\land \A\models T\}$ for some first order theory $T$, then
  denote $\ISO(\k,T)=\ISO(M)$. For a first order theory $T$, $C(\Str^\k(T))$ is Borel and
  the equivalence relation $\ISO(\k,T)$ is $\Sii$. We denote the class 
  $\{\ISO(M)\mid C(M)\text{ is Borel}\}$ by $S_\k$. The notation might be a bit confusing, since in some 
  contexts $S_\k$ denotes the group of all permutations of $\k$, but 
  note that every equivalence relation in $S_\k$ (as defined above) is induced by the action of this group.
  We choose this definition, because we will look at the class $S_\k$ as a part of the hierarchy, see below.
\end{Def}

\begin{Def}
  Given a collection of sets $\Gamma$ we say 
  that an equivalence relation $E$ on $X\subset \k^\k$ is
  \emph{$\Gamma$-complete}, if it is itself in $\Gamma$ and for every 
  equivalence relation $F\in \Gamma$ on some $Y\subset \k^\k$ there is a Borel reduction
  $F\le_B E$.

  We consider mainly $\Gamma\in\{\Borel,\Dii,\Borel^*,\Sii,S_\k\}$.
\end{Def}

Let $\dlo$ be the theory of dense linear orderings without end points.
Here is the list of results of this article:
\begin{itemize}
\item ($V=L$, $\k=\l^+$ or $\k=\aleph_\k=\l$, $\mu=\cf(\mu)<\k$) 
  The equivalence on $\l^\k$ modulo the $\mu$-non-stationary ideal is $\Sii$-complete. (Theorem~\ref{thm:Complete1})
\item ($V=L$, $\k=\l^+$, $\cf(\l)=\l$) $\ISO(\k,\dlo)$ is $\Sii$-complete. (Theorem~\ref{thm:Complete3})
\item ($\ZFC$, $\k^{<\k}=\k$) $\ISO(\k,\dlo)$ is $S_\k$-complete. (Theorem~\ref{thm:Complete4}).
\item ($V=L$, $\k=\l^+$, $\l^\o=\l$) There is a stable NDOP, NOTOP theory $T$ whose models of size $\k$ can 
  be characterized up to isomorphism by an EF-game of length $\l\cdot(\o+\o+1)$ (in particular $\ISO(\k,T)$ is $\Borel^*$ and
  $T$ is the theory of $\o+\o$ equivalence relations
  refining each other) such that $\ISO(\k,T)$ is $\Sii$-complete. (Corollary \ref{thm:Stable1}).
\item ($\k=\k^{<\k}=\l^+$, $\l^{<\l}=\l$) 
  It can be forced with a $<\k$-closed $\k^+$-c.c. forcing that 
  $\ISO(\k,T)$ for the above stable theory $T$ is not $S_\k$-complete, in fact
  $\ISO(\k,\dlo)$ is not reducible to it. (Corollary~\ref{cor:StableNotCom})
\end{itemize}


Most of the discussion in the following few pages is within $\ZFC+V=L$
(we mention it every time though).
In this theory, there is a $\Sigma_1$-formula $\f_{\le}(x,y)$ which provably
defines a well-ordering of the universe (``$\f_{\le}(x,y)\iff x\le y$'' \cite[Ch. 13]{Jech}).
By $\min_L A$ we mean the least element of $A$ in this ordering.
If $A\subset L_\theta$ is a subset of the model $L_{\theta}$ for some limit
ordinal $\theta$, then $\Sk(A)^{L_\theta}$ is the Skolem closure of
$A$ in~$L_{\theta}$ under the definable (from $\f_\le$) Skolem functions \cite[Ex. 13.24]{Jech}. 
Note that this Skolem closure $\Sk(A)^{L_\theta}\subset L_{\theta}$ is definable in $V$.
By $\ZF^-$ we mean $\ZFC+(V=L)$ without the power set axiom. 
If $\mu<\k$ is regular, then by $S^\k_\mu$ we denote all the $\mu$-cofinal ordinals less than $\k$.

\begin{Lemma}[\cite{FHK}]\label{lemma:Fri}
  Assume $V=L$. Suppose $\psi(x,\xi)$ is a $\Sigma_1$-formula in set theory with parameter $\xi\in 2^\k$
  and that $r(\a)$ is a formula of set theory that says that ``$\a$ is a regular cardinal''.
  Then for $x\in 2^\k$ we have $\psi(x,\xi)$ if and only if the set
  $$A=\{\a<\k\mid \exists \b>\a (L_{\b}\models \ZF^-\land \,\psi(x\rest\a,\xi\rest\a)\land r(\a))\}$$
  contains a cub. 
  
  Moreover ``cub'' can be replaced by $\mu$-cub for any regular $\mu<\k$.
\end{Lemma}
\begin{proof}
  Due to the length of \cite{FHK} the proof of this lemma was only sketched there, 
  so we give it here in detail.

  Suppose that $x\in 2^\k$ is such that $\psi(x,\xi)$ holds. 
  Let $\theta$ be a large enough cardinal such that
  $$L_\theta\models (\ZF^-\land \, r(\k)\land \psi(x,\xi)).$$
  For each $\a<\k$, let
  $$H(\a)=\Sk(\a\cup\{\k,\xi,x\})^{L_\theta}$$
  and $\overline{H(\a)}$ the Mostowski collapse of $H(\a)$.
  Let $$D=\{\a<\k\mid H(\a)\cap\k = \a\}.$$
  It is easy to see that $D$ is a cub set. On the other 
  hand $D\subset A$ where $A$ is as in the statement of the theorem, because each $H(\a)$ is an elementary submodel of
  $L_{\theta}$ and the Mostowski collapse $\overline{H(\a)}$ is equal to some $L_{\b}$ with $\b>\a$.
  Of course a cub set is a $\mu$-cub set for any regular $\mu<\k$.

  Suppose $x\in 2^\k$ is such that $\psi(x,\xi)$ does not hold. Similarly as above, let
  $\theta$ be a large enough cardinal such that
  $$L_\theta\models (\ZF^-\land\, r(\k)\land \lnot\psi(x,\xi))$$
  and let $C$ be a $\mu$-cub set for some regular $\mu<\k$. 
  We are going to show that $C\setminus A\ne\es$.
  Let 
  $$K(\a)=\Sk(\a\cup\{\k,C,\xi,x\})^{L_\theta}\text{ and }D=\{\a\in S^\k_\mu\mid K(\a)\cap\k=\a\}.$$
  Clearly $D$ is $\mu$-cub. Let $\a_0$ be the least ordinal in
  $\lim_\mu D$ (the set of $\mu$-cofinal limits of elements of~$D$).  
  Then we have $\a_0\in C$ by the elementarity of each $K(\a)$ and
  $$\a_0>\mu.\eqno(*)$$
  Let $\bar\b$ be
  the ordinal such that $L_{\bar\b}$ is equal to $\overline{K(\a_0)}$, the Mostowski
  collapse of $K(\a_0)$. We will show that $\a_0\notin A$ which
  completes the proof.
  Suppose on contrary, that $\a_0\in A$. 
  Then there exists $\b>\a_0$ such that
  $$L_{\b}\models \ZF^-\land\, \psi(x\rest\a_0,\xi\rest\a_0)\land r(\a_0).\eqno(**)$$ 
  This $\b$ must be a limit ordinal greater than $\bar\b$, because
  $L_{\bar\b}\models \lnot \psi(x\rest\a_0,\xi\rest\a_0)$ and $\psi$ is~$\Sigma_1$. 

  As discussed before the lemma, $K(\a)$ is a definable subset of $L_\theta$ and in fact the definition
  depends only on finitely many parameters one of which is $\a$, so also $D$ is a definable subset of $L_\theta$.
  Therefore by elementarity, $D\cap\a_0$ is a definable subset of $K(\a_0)$ and so $D\cap\a_0$ is a definable
  subset of $L_{\bar\b}$. Thus $D\cap\a_0\in L_\b$ by the definition of the $(L_\a)$-hierarchy.

  Now $L_{\b}$ satisfies $\ZF^-$ and so it satisfies 
  \begin{center}
  ``there exists a $\g\le \a_0$ and an order-preserving bijection from $\g$ to $D\cap \a_0$''.    
  \end{center}
  But there is only one such map and its domain is $\mu$, since the order-type of $D\cap\a_0$ is
  $\mu$ by the definition of $\a_0$. Hence by $(*)$ $\a_0$ is singular in $L_{\b}$ 
  which is a contradiction with~$(**)$ and the definition of~$r(\a)$.  
\end{proof}

\begin{RemarkN}\label{rem:LemmaStat}
  The following version of the lemma above can be also proved (still under $V=L$): 
  for any $\Sigma_1$-formula $\f(\eta,x)$ with parameter $x\in 2^\k$, a regular $\mu<\k$ and a stationary
  set $S\subset S^\k_\mu$, the following are equivalent for all $\eta\in 2^\k$:
  \begin{enumerate}
  \item $\f(\eta,x)$
  \item $S\setminus A$ is non-stationary, where 
    $$A=\{\a\in S\mid \exists \b>\a(L_\b\models \f(\eta\rest\a,x\rest\a)\land r(\a)\land s(\a))\},$$
    where $s(\a)$ states that $S\cap \a$ is stationary and
    $S\cap\a\subset S^\a_\mu$ in the sense that we require $\b$ to be large enough to witness
    that every element of $S\cap\a$ has cofinality $\mu$.
  \end{enumerate}
  Then the proof goes in the similar way except that we take $\a_0$ to be the smallest element
  of $(\lim_\mu D)\cap S$ instead of just $\lim_\mu D$
  and derive the contradiction in the same fashion as in the above proof
  but this time using the fact that $S$ has a lot of $\mu$-cofinal ordinals common with $D$ from
  the point of view of $L_{\b}$ which is a contradiction with the minimality of $\a_0$
  (here this $\b>\a_0$ is defined as in the proof of Lemma~\ref{lemma:Fri} to witness the counter assumption
  that~$\a_0\in A$).
\end{RemarkN}


\begin{Thm}[$V=L$]\label{thm:Complete1}
  Let $\k^{<\k}=\k>\o$. If $\k=\l^+$, let $\theta=\l$ and if $\k$ is inaccessible, let $\theta=\k$.
  Let $\mu<\k$ be a regular cardinal.
  Then the equivalence relation on $\theta^\k$ defined by 
  $$\eta\sim\xi\iff \{\a<\k\mid \eta(\a)=\xi(\a)\}\text{ contains a }\mu\text{-cub}$$
  is $\Sii$-complete.
\end{Thm}
\begin{proof}
  Suppose $E$ is a $\Sii$-equivalence relation on $\k^\k$. Let
  $a\colon \k^\k\to 2^{\k\times\k}$ be the canonical map which takes
  $\eta$ to $\xi$ such that $\xi(\a,\b)=1\iff \eta(\a)=\b$. Further
  let $b$ be a continuous bijection from $2^{\k\times\k}$ to $2^{\k}$.
  Then $c=b\circ a$ is continuous and one-to-one. Let $E'$ be the equivalence relation
  on $2^\k$ such that 
  $$(\eta,\xi)\in E'\iff (\eta=\xi)\lor(\eta,\xi\in\ran c\land \big(c^{-1}(\eta),c^{-1}(\xi)\big)\in E).$$
  Then $c$ is a continuous reduction of $E$ to $E'$. On the other hand $E'$
  is $\Sii$ because it is a continuous image of $E$ (for the generalizations of the basics
  of descriptive set theory, see~\cite{Ha,FHK}). So without loss of generality
  we can assume that $E$ is an equivalence relation on $2^\k$.
  
  For a given $\Sii$ equivalence relation $E$ on $2^\k$ and a regular $\mu<\k$
  we will define a function $f\colon 2^\k\to (2^{<\k})^\k$
  such that for all $\eta,\xi\in \k^\k$, $(\eta,\xi)\in E$ if and only if the set
  $\{\a<\k\mid f(\eta)(\a)=f(\xi)(\a)\}$ contains a $\mu$-cub and $f$ is continuous
  in the topology on $(2^{<\k})^\k$ generated by the sets
  $$\{\eta\mid \eta\rest\a = p\},\ p\in (2^{<\k})^\a,\ \a<\k.$$

  The function $f$ will be defined so that
  if $\eta\in 2^\k$, then the value of $f(\eta)$ at $\a$ will be in 
  some $L_{\g(\a)}$ where $\g(\a)<\k$ is independent
  of $\eta$. 
  If $\k=\l^+$, then for each $\g<\k$, the cardinality of $L_\g$ is at most $\l$, so
  using injections $L_\g\to \l$ it is possible to have the range~$\l^\k$. $(*)$

  If $E$ is a $\Sii$-equivalence relation, then there exists a $\Sigma_1$-formula of set theory
  $\psi(\eta,\xi)=\psi(\eta,\xi,x)=\exists k\f(k,\eta,\xi,x)$ with parameter 
  $x\in 2^\k$ which defines $E$: for all $\eta,\xi\in 2^\k$, $E(\eta,\xi)\iff \psi(\eta,\xi,x)$.

  Let $r(\a)$ be the formula that says ``$\a$ is a regular cardinal'' and let $\psi^E=\psi^E(\k)$ be the sentence
  with parameter $\k$ that asserts that $\psi(\eta,\xi)$ defines an equivalence relation on $2^\k$. 
  For $\eta\in 2^\k$ and $\a<\k$, let 
  $$T_{\eta,\a}=\{p\in 2^{\a}\mid \exists\b>\a(L_{\b}\models \ZF^-\land \psi(p,\eta\rest\a,x)\land r(\a)\land \psi^E)\}.$$
  and let
  $$f(\eta)(\a)=
  \begin{cases}
    \min{}_L T_{\eta,\a}, \text{ if }T_{\eta,\a}\ne \es,\\
    0,\text{ otherwise.}
  \end{cases}
  $$
  Note that $f(\eta)(\a)\in L_{\g}$ where $\g$ is the least ordinal
  such that $L_\g\models \lnot r(\a)$ which is $<\k$, which verifies
  the discussion above at~$(*)$.
  
  Suppose $\psi(\eta,\xi,x)=\exists k\f(k,\eta,\xi,x)$ holds
  and let $k$ be a witness of that.
  Let $\theta$ be a cardinal large enough so that $L_{\theta}\models \ZF^-\land\,\f(k,\eta,\xi,x)\land r(\k)$.
  For $\a<\k$ let $H'(\a)=\Sk(\a\cup\{\k,k,\eta,\xi,x\})^{L_\theta}$.
  Now
  $$D=\{\a<\k\mid H'(\a)\cap\k=\a\land H'(\a)\models \psi^E\}$$ 
  is a cub, and so using Mostowski-collapse we have that
  $$D'=\{\a<\k\mid \exists\b>\a (L_{\b}\models \f(k\rest\a,\eta\rest\a,\xi\rest\a,x\rest\a)
  \land \ZF^- \land\, r(\a)\land \psi^E)\}$$
  contains a cub. 
  Suppose $\a\in D'$ and $p\in T_{\eta,\a}$, i.e. 
  $$\exists\b_1>\a(L_{\b_1}\models \ZF^-\land \,\psi(p,\eta\rest\a)\land r(\a)\land\psi^E).$$
  Since $\a\in D'$, there also exists $\b_2>\a$ such that
  $$L_{\b_2}\models \ZF^-\land\, \psi(\eta\rest\a,\xi\rest\a)\land r(\a)\land\psi^E.$$
  Hence if $\b=\max\{\b_1,\b_2\}$, then 
  $$L_\b\models \psi(p,\eta\rest\a)\land \psi(\eta\rest\a,\xi\rest\a)\land \ZF^- \land r(\a)\land\psi^E$$
  and because 
  $\psi(p,\eta\rest\a)\land \psi(\eta\rest\a,\xi\rest\a)$ implies $\psi(p,\xi\rest\a)$ (because
  $\psi^E$ holds and so transitivity for $\psi(\eta,\xi)$ holds),
  we have that 
  $$L_\b\models \psi(p,\xi\rest\a)\land \ZF^-\land r(\a)\land\psi^E,$$
  which means that $p\in T_{\xi,\a}$. Thus we have proved that $T_{\eta,\a}\subset T_{\xi,\a}$.
  By symmetry we conclude $T_{\eta,\a}=T_{\xi,\a}$ and therefore $f(\eta)(\a)=f(\xi)(\a)$ for all $\a\in D'$
  which contains a cub, so this proves that
  $$\psi(\eta,\xi,x)\Rightarrow \{\a\mid f(\eta)(\a)=f(\xi)(\a)\}\text{ contains a cub.}$$

  Suppose that $\lnot\psi(\eta,\xi,x)$ holds. 
  Then by Lemma~\ref{lemma:Fri} there is no $\mu$-cub inside
  $$\{\a<\k\mid\exists\b>\a(L_{\b}\models\psi(\eta\rest\a,\xi\rest\a)\land\ZF^-\land r(\a))\},$$
  but this is a superset of
  $$\{\a<\k\mid\exists\b>\a(L_{\b}\models\psi(\eta\rest\a,\xi\rest\a)\land\ZF^-\land r(\a))\land \psi^E\},\eqno(**)$$
  so the latter does not contain a $\mu$-cub either.
  Now
  \begin{eqnarray*}
     &&\{\a\mid f(\eta)(\a)=f(\xi)(\a)\}\\
    &=&\{\a\mid \min{}_L T_{\eta,\a}=\min{}_L T_{\xi,\a}\}\\
    &\subset&\{\a\mid\exists p\in T_{\eta,\a}\cap T_{\xi,\a}\}\\
    &=&\{\a\mid\exists p\exists\b_1,\b_2>\a\big((L_{\b_1}\models \psi(p,\eta\rest\a)\land  \ZF^-\land r(\a)\land\psi^E)\\
    &&\phantom{\{\a\mid\exists p\exists\b_1,\b_2>}\land (L_{\b_2}\models \psi(p,\xi\rest\a)\land \ZF^-\land r(\a)\land \psi^E)\big)\}
  \end{eqnarray*}
  ... and taking $\b=\max\{\b_1,\b_2\}$ we continue:
  \begin{eqnarray*}
    &\subset&\{\a\mid\exists p\exists\b>\a(L_{\b}\models \psi(p,\eta\rest\a)\land\psi(p,\xi\rest\a)\land \ZF^-\land r(\a)\land\psi^E)\}\\
    &=&\{\a\mid \exists\b>\a(L_{\b}\models \psi(\eta\rest\a,\xi\rest\a)\land \ZF^-\land r(\a)\land\psi^E)
  \end{eqnarray*}
  which by $(**)$ does not contain $\mu$-cub, so $\{\a<\k\mid f(\eta)(\a)=f(\xi)(\a)\}$ doesn't contain one either.
\end{proof}

\begin{Remark}
  By using the modified version of Lemma~\ref{lemma:Fri} as described in Remark~\ref{rem:LemmaStat}
  one can prove a stronger result. Let $\l,\k$ and $\theta$ be as in the theorem above and $\mu<\k$ regular. 
    For every stationary $S\subset S^\k_\mu$ the equivalence relation on $\theta^\k$ defined by
    $$\eta\sim\xi\iff S\setminus\{\a\mid\eta(\a)=\xi(\a)\}\text{ is non-stationary}$$
    is $\Sii$-complete.
\end{Remark}

Now we will use Theorem \ref{thm:Complete1} first to show that the isomorphism relations $\ISO(\k,\dlo)$ and $\ISO(\k,T_{\o+\o})$ 
are $\Sii$-complete. Both within $\ZFC+V=L$.

\begin{Def}[Colored Linear Orders]
  A \emph{colored linear order} (\emph{clo}) is a pair $(L,c)$ where $L$ is a linear order and $c$ is a function with domain $L$.
  An isomorphism between clos $(L,c)$ and $(L',c')$ is a function $f\colon L\to L'$ which is an isomorphism
  between $L$ and $L'$ and preserves coloring: $c(x)=c'(f(x))$. 
  If $(L,c)$ and $(L',c')$ are clos, then $(L,c)+(L',c')$ is the clo $(L+L',d)$, where $d$
  is such that $d\rest L=c$ and $d\rest L'=c'$. Similarly, if $L'$ is any linear order and $(L,c)$ is a clo,
  then $(L,c)\cdot L'$ is the clo $(L\cdot L',d)$, where $d\rest L\cdot \{x\}=c$ for any $x\in L'$.
\end{Def}

\begin{Thm}[$V=L$] \label{thm:Complete3}
  Suppose $\k=\l^+$ and $\l$ is regular.
  The isomorphism relation on the class of dense linear orderings of size $\k$ 
  is $\Sii$-complete. If $\l>\o$, one can assume that all the orderings are $\k$-like, i.e. 
  all initial segments have size $<\k$.
\end{Thm}
\begin{proof}
  We will show that there exists a continuous function $f\colon\l^\k\to 2^\k$ such that
  for all $\eta\in \l^\k$, $\A_{f(\eta)}$ is a dense linear order without end points and
  for all $\eta,\xi\in\l^{\k}$ the set
  $\{\a<\k\mid \eta(\a)=\xi(\a)\}$ contains a $\l$-cub if and only if
  $\A_{f(\eta)}\cong \A_{f(\xi)}$. 
  Thus we embed a $\Sii$-complete (by Theorem~\ref{thm:Complete1}) equivalence relation into the isomorphism of dense linear orders
  which suffices.

  We will first define a function $f$ which attaches to each function in $\l^\k$ a colored linear order. Then
  we will show how to eliminate the use of colors by replacing each point by a linear ordering which depends on its color.
  
  Let $\eta$ be a saturated dense linear ordering without end points
  of cardinality $\l$. Suppose that the coloring $c\colon \eta\to\l\setminus\{0\}$ satisfies 
  \begin{itemize}
  \item[$(*)$] If $A,B\subset\eta$ have cardinality
    less than $\l$ and $x\in A,y\in B$ we have $x<y$, then for all $\a\in\l$ there exists $z$ with $x<z<y$ for all $x\in A, y\in B$
    and $c(z)=\a$.
  \end{itemize}
  Then we call $(\eta,c)$ \emph{a saturated clo}.
  \begin{claim}{1}
    A saturated clo exists.
  \end{claim}
  \begin{proofVOf}{Claim 1}
    This can be done for example as follows. Let $\xi$ be a saturated dense linear order with domain $\l\setminus\{0\}$.
    Let $\eta=\{f\colon\a+1\to \xi\mid \a<\l\}$ and for $f,g\in\eta$ let
    $$f<g\iff (f\subset g)\lor (f(\a)<g(\a)\text{, where }\a=\min\{\b\mid f(\b)\ne g(\b)\}).$$
    Let $c(f)=f(\max\dom(f))$.
    It is not difficult to check that this satisfies all the requirements.
  \end{proofVOf}

  Given two saturated colored linear orderings $(\eta,c)$ and $(\eta',c)$,
  there is an isomorphism $f\colon (\eta,c)\to(\eta',c')$.
  This can be seen by a simple back-and-forth argument. 
  By this observation we have: if $\eta=(\eta,c)$ is a saturated clo, then 
  \begin{itemize}
  \item[(1)] $\eta\cong \eta+\eta$,
  \item[(2)] for all $\a<\l$, $\eta\cong \eta + (1,c_\a) +\eta $, where $1$ is a linear ordering of length $1$ and $c_\a$ is the coloring with 
    range $\{\a\}$,
  \item[(3)] for all $\a<\k$, $\eta\cdot \a + \eta\cong \eta$,
  \item[(4)] if $\a<\k$ is $\l$-cofinal and 
    $\tau_i=(1,c_{\b_i})+\eta$ for all $\l$-cofinal $i<\a$ and $\tau_i=\eta$ otherwise, then 
    $(\sum_{i<\a}\tau_i)\cong\eta$.
  \end{itemize}
  
  We will now define a function with domain $\l^\k$ and range the set of colored linear orders.
  Then we will define a function from that range into (non-colored) dense linear orders.

  As above, denote by $(1,c_\a)$ a clo with a single element of color $\a$.
  Given $f\colon \k\to\l$, let
  $$\Phi(f)=\sum_{\a<\k}\tau_{\a}^f,$$
  where $\tau_\a^f=(1,c_{f(\a)})+\eta$, if $\cf(\a)=\l$ and $\tau^f_\a=\eta$ otherwise.
  
  \begin{claim}{2}
    \nopagebreak{For $f,g\in \l^\k$, the set $\{\a\mid f(\a)=g(\a)\}$ contains a $\l$-cub if and only if 
    $\Phi(f)\cong \Phi(g)$.}
  \end{claim}
  \begin{proofVOf}{Claim 2}
    Suppose $\{\a\mid f(\a)=g(\a)\}$ contains a $\l$-cub $C\subset S^\k_\l$. 
    Let $\{a_i\mid i<\k\}$ be an enumeration of $C$ such that $i<j\iff a_i<a_j$.
    The orderings
    $$\Phi_{i}(f)=\sum_{a_i\le\a<a_{i+1}}\tau_{\a}^f$$
    and 
    $$\Phi_{i}(g)=\sum_{a_i\le\a<a_{i+1}}\tau_{\a}^g$$
    are isomorphic by (4) above and because the color of $\min \Phi_{i}(f)$ is the same as that
    of $\min \Phi_i(g)$ for all $i$ by the definition of $C$. This proves ``$\Rightarrow$''.

    Suppose $\{\a\mid f(\a)=g(\a)\}$ does not contain a $\l$-cub. Then its complement
    contains a $\l$-stationary set $S\subset S^\k_\l$. Suppose for a contradiction that 
    there is an isomorphism $F$ between $\Phi(f)$ and $\Phi(g)$.
    Let
    $$\Phi^i(f)=\sum_{\a<i}\tau_{\a}^f$$
    and
    $$\Phi^i(g)=\sum_{\a<i}\tau_{\a}^g.$$
    By the standard argument, there is a cub set $C$ such that for all $i\in C$ we have that
    the restriction
    $F\rest\Phi^i(f)\colon \Phi^i(f)\to \Phi^i(g)$ is an isomorphism. Pick
    and element $j\in C\cap S$. Then $F\rest\Phi^j(f)$ is an isomorphism, but
    the color of $\min (\Phi(f)\setminus \Phi^j(f))$ is not the same as the color
    of $\min (\Phi(g)\setminus \Phi^j(g))$ which is a contradiction.    
  \end{proofVOf}
  
  Denote by $\ISO(\clo)$ the isomorphism relation on all colored linear orders
  and let $\ISO(\dlo)$ be the isomorphism relation on all the (non-colored) dense linear orders.
  We have shown now that an arbitrary $\Sii$-equivalence relation $E$ is Borel reducible to
  $\ISO(\clo)$. Next we show how to reduce the range of that reduction to $\ISO(\dlo)$.

  To do that, we replace every point of $\Phi(f)$ by a (non-colored) dense linear order whose isomorphism
  type depends on the color of the corresponding point
  and we will do it so that the original clos are isomorphic if and only if the resulting dense linear orders
  are isomorphic. Suppose first that $\l>\o$.
  Let $(S_i)_{i<\l}$ be a $\l$-long sequence of disjoint stationary subsets of $\l$. 
  Let 
  $$\xi_i=\sum_{\a<\k}\s_\a,$$
  where $\s_\a=1+\Q$, if $\a\in S_i$ and $\s_\a=\Q$ otherwise, where $\Q$ is the order of the rational numbers.
  Then $\xi_i\not\cong\xi_{j}$ for all $i\ne j$ by a similar argument as above for $\Phi(f)\not\cong \Phi(g)$.
  Then let 
  $$\Psi(f)=\sum_{a\in \Phi(f)}\xi_{c(a)}.$$
  Note that since $\Phi(f)$ is $\k$-like, also $\Psi(f)$ is $\k$-like.

  If $\l=\o$, then do the same, but now $S_i$ are stationary subsets of $\k=\o_1$.
  In this case we lose the property that $\Psi(f)$ is $\k$-like.
\end{proof}

\section{The Isomorphism Relation of $\M(T_{\o+\o})$}

The models of $T_{\o+\o}$ which we are going to investigate are essentially certain 
trees as will be shown later (Lemma~\ref{lemma:trTbired}). Thus we show first that the
isomorphism relation on these trees is $\Sii$-complete in $L$ by using results from previous section.
The we will show, using a result from \cite{HK}, that it is consistent that the isomorphism relation 
is not $\Sii$-complete and in fact not even $S_\k$-complete.

In \cite{FHK} assuming $\k=\l^+$ and $\l=\l^\o$, we constructed
for each set $S\subset S^\k_\o$ a $\k^+,(\o+2)$-tree
$J(S)$ such that $S\sd S'$ is non-stationary $\iff$ $J(S)\cong J(S')$. Adopting
a similar construction for colored trees, we will construct for each function
$f\in \l^\k$ a colored $\k^+,(\o+2)$-tree $J_f$ such that for all $f,g\in \l^\k$, the set
$$\{\a\mid f(\a)=g(\a)\}$$
contains an $\o$-cub if and only if $J_f\cong J_g$. The proof of Lemma 4.89 of \cite{FHK}
has to be modified such that instead of the \emph{dichotomy} for every branch to either have a leaf or not, 
there is a \emph{$\l$-chromatomy} for the leaf of every branch to be of one of the $\l$ colors.

\begin{Def}\label{def:CT}
  Suppose $\k=\l^+$.
  A \emph{colored $\k^+,(\o+2)$-tree} is a pair $(t,c)$, where $t$ is a $\k^+,(\o+2)$-tree
  (every element has less than $\k^+$ successors and all branches have
  order type less than $\o+2$) and $c$ is a map whose domain is the set $\{x\in t\mid \hgt(x)=\o\}$,
  where $\hgt(x)$ is the height of $x$ -- the order type of $\{y\in t\mid y<x\}$.
  The range of $c$ is $\l\setminus \{0\}$.

  An isomorphism between colored trees $(t,c)$ and $(t',c')$ is a map $f\colon t\to t'$ which
  is an isomorphism between $t$ and $t'$ and for each $x\in \dom c$, $c(x)=c'(f(x))$.
  
  Denote the set of all colored $\k^+,(\o+2)$-trees by $\CT^\o$.
  Let $\CT^\o_*\subset \CT^\o$ be the set of those trees in which
  every element has infinitely many successors at each level $<\o$.
\end{Def}

\begin{Def}\label{def:Filtration}
  Let $t$ be a colored tree of size $\k=\l^+$. Suppose $(I_{\a})_{\a<\k}$ is a collection of subsets of $t$ such that
  \begin{myItemize}
  \item for each $\a<\k$, $I_\a$ is a downward closed subset of $t$,
  \item $\Cup_{\a<\k} I_{\a}=t$,
  \item if $\a<\b<\k$, then $I_{\a}\subset I_{\b}$,
  \item if $\g$ is a limit ordinal, then $I_\g=\Cup_{\a<\g}I_\a$,
  \item for each $\a<\k$ the cardinality of $I_\a$ is less than $\k$.
  \end{myItemize}
  Such a sequence $(I_{\a})_{\a<\k}$ is called $\k$-\emph{filtration} or just \emph{filtration} of $t$. 
\end{Def}

\begin{Def}\label{def:FuncT}
  Suppose that $\k=\l^+$ and $(t,c)$ is a colored tree of size $\k$, $t\subset\k^{\le\o}$, 
  with colors ranging in $\l\setminus \{0\}$ and let $\I=(I_{\a})_{\a<\k}$ be a filtration of $t$.
  Define $f_{\I,t}\in \l^\k$ as follows. Fix $\a<\k$.
  Let $B_\a$ be the set of all $x\in t$ with $x\notin I_\a$, but $x\rest n\in I_\a$ for all $n<\o$. 
  \begin{itemize}
  \item[(a)] If $B_\a$ is non-empty and there is $\b$ 
    such that for all $x\in B_\a$, $c(x)=\b$, then let $f_{\I,t}(\a)=\b$,
  \item[(b)] Otherwise let $f_{\I,t}(\a)=0$
  \end{itemize}
\end{Def}

For $f,g\in\l^\k$, by $f\sim g$ we mean that 
$\{\a<\k\mid f(\a)=g(\a)\}$ contains an $\o$-cub.

\begin{Lemma}\label{lemma:FiltrEquiv}
  Suppose colored trees $(t_0,c_0)$ and $(t_1,c_1)$ are isomorphic, and
  $\I=(I_\a)_{\a<\k}$ and $\J=(J_\a)_{\a<\k}$ are $\k$-filtrations of $t_0$ and $t_1$ respectively.
  Then $f_{\I,t_0}\sim f_{\J,t_1}$.  
\end{Lemma}
\begin{proof}
  Let $F\colon t_0\to t_1$ be a (color preserving) isomorphism. Then 
  $F\I=(F[I_\a])_{\a<\k}$ is a filtration of $t_1$ and for all $\a<\k$
  $$f_{\I,t_0}(\a)=f_{F\I,t_1}(\a).\eqno(\star)$$
  Define the set $C=\{\a\mid F[I_\a]=J_\a\}$. Let us show that it is $\o$-cub. 
  Let $\a\in \k$. Define $\a_0=\a$ and by induction pick $(\a_n)_{n<\o}$ 
  such that $F[I_{\a_n}]\subset J_{\a_{n+1}}$ for odd $n$ and $J_{\a_n}\subset F[I_{\a_{n+1}}]$ for even $n$.
  This is possible by the definition of a $\k$-filtration. Then $\a_\o=\Cup_{n<\o}\a_n\in C$.
  Clearly $C$ is closed and 
  $C\subset \{\a<\k\mid f_{F\I,t_1}(\a)= f_{\J,t_1}(\a)\}$ so now by $(\star)$ we have the result.
\end{proof}

\begin{Lemma}\label{lem:StoJS}
  Suppose $\k=\l^+$, $\l^\o=\l$ and $\k^{<\k}=\k$. 
  There exists a function $J\colon \l^\k\to \CT^\o_*$ such that for all $f,g\in \l^\k$,
  $$f\sim g\iff J_f\cong J_g$$
  (as colored trees).
\end{Lemma}
\begin{proofV}{Lemma \ref{lem:StoJS}}  
  Define the ordering on $\o\times\k\times\k\times\k\times\k$ lexicographically, i.e.
  such that $(\a_1,\a_2,\a_3,\a_4,\a_5)<_{\text{lex}}(\b_1,\b_2,\b_3,\b_4,\b_5)$ if and only if
  $\a_k<\b_k$ for the smallest $k$ with $\a_k\ne \b_k$ (and such $k$ exists).
  We order the set $(\o\times\k\times\k\times\k\times\k)^{\le \o}$ as a tree:
  $\eta<\xi$ if and only if $\eta\subset\xi$.

  For each $f\in \l^\k$ we will define a colored tree $J_f=(J_f,c_f)$ such that 
  $J_f\subset (\o\times\k\times\k\times\k\times\k)^{\le\o}$ with the induced ordering
  and
  \begin{myItemize}
  \item[(a)] If $f\in \l^\k$ and $\I$ is any $\k$ filtration of $J_f$, then $f_{\I,J_f}\sim  f$.
  \item[(b)] If $f\sim  g$, then $J_f\cong J_g$.
  \end{myItemize}
  This suffices, because if $J_f\cong J_g$, then 
  for some filtrations $\I$ and $\J$ of $J_f$ and $J_g$ respectively we have
  by Lemma \ref{lemma:FiltrEquiv} that $f_{\I,J_f}\sim f_{\J,J_g}$ which further implies by (a)
  that $f\sim g$.

  Let $f\in\l^\k$ and let us define a preliminary colored tree $(I_f,d_f)$
  as follows. Let $I_f$ be the tree of all strictly increasing functions 
  from $n\le\o$ to $\k$ and for $\eta$ with domain $\o$, let $d_f(\eta)=f(\sup\ran \eta)$.

  For ordinals $\a<\b$ and $i<\o$ we adopt the notation:
  \begin{myItemize}
  \item $[\a,\b]=\{\g\mid \a\le \g\le\b\}$,
  \item $[\a,\b)=\{\g\mid \a\le \g <\b\}$,
  \item $\ruszh(\a,\b,i)=\Cup_{i\le j\le\o}\{\eta\colon [i,j)\to [\a,\b)\mid \eta\text{ strictly increasing}\}$.
  \end{myItemize}

  For each $\a,\b<\k$ let us define the colored trees $P^{\a,\b}_{\g}$, for $\g<\k$ as follows. 
  If $\a=\b=\g=0$, then $P^{0,0}_{0}=(I_f,d_f)$. Otherwise let $\{P^{\a,\b}_\g\mid \g<\k\}$
  enumerate all downward closed subtrees of $\ruszh(\a,\b,i)$ for all $i$, with all possible colorings
  and in such a way, that every isomorphism type appears cofinally often in the enumeration.
  The isomorphism types are of course counted with respect to color preserving isomorphisms.
  Define 
  $$\rusch(P^{\a,\b}_{\g})$$ 
  to be the natural number $i$ such that $P^{\a,\b}_{\g}\subset \ruszh(\a,\b,i)$.
  The enumeration is possible, because the number of all downward closed subsets
  of $\ruszh(\a,\b,i)$, $i<\o$, is at most
  \begin{eqnarray*}
  \Big|\Cup_{i<\o}\Po(\ruszh(\a,\b,i))\Big|&\le& \o\times |\Po(\ruszh(0,\b,0))|\\    
  &\le&\o\times |\Po(\b^\o)|\\
  &=&\o\times 2^{\b^{\o}}\\
  &\le&\o\times \k\\
  &=&\k
  \end{eqnarray*}
  and since for each $\b<\k$, $\ruszh(\a,\b,i)$ has cardinality $<\k$, even when we add all possible colorings,
  the number of trees remains $\k\times \l^\l=\k$.
  For $f\in \l^\k$ define $J_f=(J_f,c_f)$ to be the tree of all 
  $\eta\colon s\to \o\times \k\times\k\times\k\times\k=\o\times\k^4$ 
  such that $s\le\o$ and the following conditions are met for all $i,j<s$:
  \begin{myEnumerate}
  \item \label{J-1}$\eta\rest n\in J_f$ for all $n<s$,
  \item \label{J0}$\eta$ is strictly increasing with respect to the lexicographical order on $\o\times\k^4$,
  \item \label{J1}$\eta_1(i)\le \eta_1(i+1)\le \eta_1(i)+1$,
  \item \label{J3}$\eta_1(i)=0\rightarrow \eta_{2}(i)=\eta_{3}(i)=\eta_{4}(i)=0$,
  \item \label{J4}$\eta_1(i)<\eta_1(i+1)\rightarrow \eta_2(i+1)\ge \eta_3(i)+\eta_4(i)$,
  \item \label{J5}$\eta_1(i)=\eta_1(i+1)\rightarrow (\forall k\in\{2,3,4\})(\eta_k(i)=\eta_k(i+1))$,
  \item \label{J6}if for some $k<\o$, $[i,j)=\eta_1^{-1}\{k\}$, then\\
    $\eta_5\rest[i,j)\in P^{\eta_2(i),\eta_3(i)}_{\eta_4(i)}$
  \item \label{J7}if $s=\o$, then either 
    \begin{itemize}
   \item[(a)] $(\exists m<\o)(\forall k<\o)(k>m\rightarrow \eta_1(k)=\eta_1(k+1))$ and the color of $\eta$
    is determined by $P^{\eta_2(m),\eta_3(m)}_{\eta_4(m)}$: $c_f(\eta)=c(\eta_5)$, where $c$ is the coloring
    of $P^{\eta_2(m),\eta_3(m)}_{\eta_4(m)}$.
    \item[or else]
    \item[(b)] $c_f(\eta)=f(\sup\ran \eta_5).$
    \end{itemize}
  \item Order $J_f$ as a subtree of $(\o\times\k^4)^{\le \o}$, $\eta<\xi\iff\eta\subset\xi$.
  \end{myEnumerate}

  Note that it follows from the definition of $P^{\a,\b}_\g$ and the conditions \eqref{J6} and
  \eqref{J4} that for all $i<j<\dom \eta$ and $\eta\in J_f$:\\

  \begin{myEnumerate}\setcounter{enumi}{9}
  \item \label{J2} $i<j\rightarrow \eta_5(i)<\eta_5(j)$.\\
  \end{myEnumerate}
  Also we have that if $\eta\in (\o\times\k^4)^{\le \o}$ is such that
  $\eta\rest n\in J_f$ for all $n$, then $\eta\in J_f$.

  It is easy to see that these trees are in $\CT^\o_*$.

  For each $\a<\k$ let 
  $$J^{\a}_f=\{\eta\in J_f\mid \ran\eta\subset\o\times(\b+1)^4\text{ for some }\b<\a\}.$$
  Then $(J^{\a}_f)_{\a<\k}$ is a $\k$-filtration of $J_f$ (see Claim~2 below). 

  If $\eta\in J_f$ and $\ran\eta_1=\o$, then 
  $$\sup\ran\eta_4\le\sup\ran\eta_2=\sup\ran\eta_3=\sup\ran\eta_5\eqno(\#)$$
  and if in addition to that, $\eta\rest k\in J^{\a}_f$ 
  for all $k$ and $\eta\notin J^{\a}_f$ or if $\ran\eta_1=\{0\}$, then 
  $$\sup\ran\eta_5=\a.\eqno(\circledast)$$\label{circledast}
  To see $(\#)$ suppose $\ran\eta_1=\o$.  By \eqref{J2}, $(\eta_5(i))_{i<\o}$ is an
  increasing sequence. By \eqref{J6} $\sup\ran\eta_3\ge\sup\ran\eta_5\ge\sup\ran\eta_2$.
  By \eqref{J4}, $\sup\ran\eta_2\ge\sup\ran\eta_3$ and again by \eqref{J4} $\sup\ran\eta_2\ge \sup\ran\eta_4$.
  Inequality $\sup\ran\eta_5\le\a$ is an immediate consequence of the definition of $J^{\a}_f$,
  so $(\circledast)$ follows now from the assumption that $\eta\notin J^{\a}_f$.

  \begin{claim}{1}
    Suppose $\xi\in J^{\a}_f$ and $\eta\in J_f$. Then
    if $\dom \xi<\o$, $\xi\subsetneq \eta$ and 
    $(\forall k\in \dom\eta\setminus\dom\xi)\big(\eta_1(k)=\xi_1(\max\dom\xi)\land \eta_1(k)>0\big),$
    then
    $\eta\in J^{\a}_f$.
  \end{claim}
  \begin{proofVOf}{Claim 1}
    Suppose $\xi,\eta\in J^{\a}_f$ are as in the assumption. Let us define 
    $\b_2=\xi_2(\max\dom\xi)$, $\b_3=\xi_2(\max\dom\xi)$, and
    $\b_4=\xi_4(\max\dom\xi)$. Because $\xi\in J^{\a}_f$, there is $\b$ such that
    $\b_2,\b_3,\b_4<\b+1$ and $\b<\a$.
    Now by \eqref{J5} $\eta_2(k)=\b_2$, $\eta_3(k)=\b_3$ and $\eta_4(k)=\b_4$, for all $k\in \dom\eta\setminus\dom\xi$. 
    Then by \eqref{J6} for all 
    $k\in\dom\eta\setminus\dom\xi$ we have that 
    $\b_2<\eta_5(k)<\b_3<\b+1$. Since $\xi\in J^{\a}_f$,
    also $\b_4<\b+1$, so $\eta\in J^{\a}_f$. 
  \end{proofVOf}

  \begin{claim}{2}
    $|J_f|=\k$, $(J^\a_f)_{\a<\k}$ is a $\k$-filtration of $J_f$
    and if $f\in\l^\k$ and $\I$ is a $\k$-filtration of $J_f$, then $f_{\I,J_f}\sim  f$.
  \end{claim}
  \begin{proofVOf}{Claim 2}
    For all $\a<\k$,
    $J^\a_f\subset (\o\times\a^{4})^{\le\o}$, so by the cardinality assumption of the lemma, 
    the cardinality of $J^\a_f$ is $<\k$.
    Clearly $\a<\b$ implies $J^{\a}_f\subset J^{\b}_f$. Continuity is verified by
    \begin{eqnarray*}
      \Cup_{\a<\g}J^{\a}_f&=&\{\eta\in J_f\mid\exists\a<\g,\exists\b<\a(\ran\eta\subset \o\times (\b+1)^4)\}\\
      &=&\{\eta\in J_f\mid\exists\b<\cup\g(\ran\eta\subset \o\times (\b+1)^4)\}
    \end{eqnarray*}
    which equals $J^\g_f$ if $\g$ is a limit ordinal.

    By Lemma \ref{lemma:FiltrEquiv} it is enough to show that $f_{\I,J_f}\sim  f$ for
    $\I=(J^\a_f)_{\a<\k}$, and we will show that if $\I=(J^\a_f)_{\a<\k}$, then 
    for all $\o$-cofinal ordinals $\a$ we have $f_{\I,J_f}(\a)=f(\a)$.

    Suppose $\a$ is $\o$-cofinal and suppose that $\eta$ is such that
    $\eta\notin J^\a_f$, $\eta\rest k\in J^\a_f$, $k<\o$. By Claim 1 $\eta$ can
    satisfy (a) of \eqref{J7} only if $\eta_1(n)=0$ for all $n<\o$.
    In that case $\eta=(0,0,0,0,\eta_5)$ and $\eta_5$ is in $P^{0,0}_0=I_f$ 
    and by the definition in \eqref{J7}(a) we have $c_f(\eta)=d_f(\eta_5)$, 
    which is by definition $d_f(\eta_5)=f(\sup\ran \eta_5)=f(\a)$ (see the definition of
    $P^{0,0}_0$ and $(I_f,d_f$) above. 

    Else, if (b) of \eqref{J7} is satisfied, then again $c_f(\eta)=f(\sup\ran\eta_5)$
    which is by  $(\circledast)$ equal to $f(\a)$. So that means that the color of all such 
    $\eta$ is $f(\a)$ and thus in defining $f_{\I,J_f}(\a)$ we use
    the condition (b) of Definition~\ref{def:FuncT} and get $f_{\I,J_f}(\a)=f(\a)$.
  \end{proofVOf}

  \begin{claim}{3}
    Suppose $f\sim g$. Then $J_f\cong J_g$.  
  \end{claim}
  \begin{proofVOf}{Claim 3}
    Let $C'\subset \{\a<\k\mid f(\a)=g(\a)\}$ be the $\o$-cub set which exists by the assumption
    and let $C$ be its closure under limits of uncountable cofinality. We will build a back-and-forth system
    to find the isomorphism. By induction on $i<\k$ we will define $\a_i$ and $F_{\a_i}$ such that:
    \begin{myAlphanumerate}
    \item If $i<j<\k,$ then $\a_i<\a_j$ and $F_{\a_i}\subset F_{\a_j}$.
    \item If $i$ is a successor, then $\a_i$ is a successor and if $i$ is limit, then $\a_i\in C$.
    \item If $\g$ is a limit ordinal, then $\a_\g=\sup_{i<\g}\a_i$.
    \item $F_{\a_i}$ is a color preserving partial isomorphism $J_f\to J_g$.
    \item Suppose that $i=\g+n$, where $\g$ is a limit ordinal or $0$ and $n<\o$ is even. Then
      $\dom F_{\a_i}=J^{\a_i}_f$.
    \item If $i=\g+n$, where $\g$ is a limit ordinal or 0 and $n<\o$ is odd, then
      $\ran F_{\a_i}=J^{\a_i}_g$. 
    \item If $\dom \xi<\o$, $\xi\in \dom F_{\a_i}$,
      $\eta\rest\dom\xi=\xi$ and $(\forall k\ge\dom\xi)\big(\eta_1(k)=\xi_1(\max\dom\xi)\land \eta_1(k)>0\big)$, then
      $\eta\in \dom F_{\a_i}$. Similarly for $\ran F_{\a_i}$.
    \item If $\xi\in\dom F_{\a_i}$ and $k<\dom \xi$, then $\xi\rest k\in \dom F_{\a_i}$.
    \item For all $\eta\in \dom F_{\a_i}$, $\dom\eta=\dom (F_{\a_i}(\eta))$.
    \end{myAlphanumerate}
    
    \noindent\textbf{The first step.} 
    The first step and the successor steps are similar, but the first step is easier. Thus we give it separately
    in order to simplify the readability.
    Let us start with $i=0$. Let $\a_0=\b+1$, for arbitrary $\b\in C$. Let us denote by 
    $$\rusi(\a)$$
    the ordinal $\o\cdot\a^4$ that is order isomorphic to $\o\times\a^4\subset \o\times\k^4$ (the order
    on the latter is defined in the beginning of this section).
    Let $\g$ be such that there is a (color preserving) isomorphism $h\colon P^{0,\rusi(\a_0)}_{\g}\cong J^{\a_0}_f$ 
    and such that $\rusch(P^{0,\a_0}_{\g})=0$. Such exists by \eqref{J0}.
    Suppose that $\eta\in J^{\a_0}_f$. Note that because $P^{0,\a_0}_{\g}$ and $J^{\a_0}_f$ are closed 
    under initial segments and by the definitions of $\rusch$ and $P^{\a,\b}_\g$, we have $\dom h^{-1}(\eta)=\dom \eta$.
    Define $\xi=F_{\a_0}(\eta)$ such that $\dom\xi=\dom\eta$ 
    and for all $k<\dom \xi$
    \begin{myItemize}
      \item $\xi_1(k)=1$,
      \item $\xi_2(k)=0$,
      \item $\xi_3(k)=\rusi(\a_0)$,
      \item $\xi_4(k)=\g$,
      \item $\xi_5(k)=h^{-1}(\eta)(k)$.
    \end{myItemize}
    Let us check that $\xi\in J_g$. Conditions \eqref{J0}-\eqref{J5} and \eqref{J7} are satisfied because
    $\xi_k$ is constant for all $k\in \{1,2,3,4\}$, $\xi_1(i)\ne 0$ for all $i$ and $\xi_5$ is increasing. For \eqref{J6}, if
    $\xi_1^{-1}\{k\}$ is empty, the condition is verified since each $P^{\a,\b}_\g$ is closed under initial segments 
    and contains the empty function. If it is non-empty, then $k=1$ and in that case $\xi_1^{-1}\{k\}=[0,\o)$ and
    by the argument above ($\dom h^{-1}(\eta)=\dom \eta=\dom\xi$) we have
    $\xi_5=h^{-1}(\eta)\in P^{0,\rusi(\a_0)}_\g=P^{\xi_2(0),\xi_3(0)}_{\xi_4(0)}$, so the condition is satisfied.
    The colors are preserved, because $h$ is an isomorphism.
    
    Let us check whether all the conditions (a)-(i) are met. In (a), (b), (c) 
    and (f) there is nothing to check.
    (d) holds, because $h$ is an isomorphism. (e) and (i) are immediate from the definition. 
    Both $J^{\a_0}_f$ and $P^{0,\rusi(\a_0)}_\g$
    are closed under initial segments, so (h) follows, because $\dom F_{\a_0}=J^{\a_0}_f$ and 
    $\ran F_{\a_0}=\{1\}\times \{0\}\times \{\rusi(\a_0)\}\times \{\g\}\times P^{0,\a_0}_{\g}$.
    Claim~1 implies (g) for $\dom F_{\a_0}$. Suppose $\xi\in \ran F_{\a_0}$ and $\eta\in J_g$ are as in the assumption of (g).
    Then $\eta_1(i)=\xi_1(i)=1$ for all $i<\dom \eta$. By \eqref{J5} it follows that 
    $\eta_2(i)=\xi_2(i)=0$,
    $\eta_3(i)=\xi_3(i)=\rusi(\a_0)$ and
    $\eta_4(i)=\xi_4(i)=\g$
    for all $i<\dom \eta$, so by \eqref{J6} $\eta_5\in P^{0,\rusi(\a_0)}_\g$ and since $h$ is an isomorphism,
    $\eta\in\ran F_{\a_0}$.\\

    \noindent\textbf{Odd successor step.}
    We want to handle odd case first, 
    because the most important case is the successor of a limit ordinal, see $(\iota\iota\iota)$
    below. Except that, the even case is similar to the odd case.

    Suppose that $j<\k$ is a successor ordinal. Then there exist $\b_j$ and $n_j$ such that
    $j=\b_j+n_j$ and $\b_j$ is a limit ordinal or $0$. Suppose that $n_j$ is odd and
    that $\a_l$ and $F_{\a_l}$ are defined for all $l<j$ such that the conditions (a)--(i) and \eqref{J0}--\eqref{J2}
    hold for $l<j$.
    
    Let $\a_j=\b+1$ where $\b$ is such that $\b\in C$, $\ran F_{\a_{j-1}}\subset J^{\b}_g$ and $\b>\a_{j-1}$.
    For convenience define $\xi(-1)=(0,0,0,0,0)$ for all $\xi\in J_f\cup J_g$.
    Suppose $\eta\in \ran F_{\a_{j-1}}$ has finite domain $\dom\eta=m<\o$ and denote $\xi=F^{-1}_{\a_{j-1}}(\eta)$. 
    Fix $\g_\eta$ to be such that $\rusch(P^{\a,\b}_{\g_\eta})=m$ and
    such that there is an isomorphism 
    $h_\eta\colon P^{\a,\b}_{\g_\eta}\to W,$
    where 
    $$W=\{\zeta\mid \dom\zeta=[m,s), m<s\le\o, 
       \eta^{\frown}\la m,\zeta(m)\ra\notin 
       \ran F_{\a_{j-1}}, \eta^{\frown}\zeta\in J^{\a_j}_g\},$$
    $\a=\xi_3(m-1)+\xi_4(m-1)$ and $\b=\a+\rusi(\a_j)$ (defined in the beginning of the First step).

    We will define $F_{\a_{j}}$ so that its range is $J^{\a_{j}}_g$ and instead of $F_{\a_j}$ we will
    define its inverse. 
    So let $\eta\in J^{\a_j}_g$. We have three cases:
    \begin{myItemize}
    \item[($\iota$)] $\eta\in \ran F_{\a_{j-1}}$,
    \item[($\iota\iota$)] $\exists m<\dom\eta(\eta\rest m\in \ran F_{\a_{j-1}}\land \eta\rest(m+1)\notin F_{\a_{j-1}})$,
    \item[($\iota\iota\iota$)] $\forall m<\dom\eta(\eta\rest(m+1)\in \ran F_{\a_{j-1}}\land \eta\notin \ran F_{\a_{j-1}})$.
    \end{myItemize}
    Let us define $\xi=F^{-1}_{\a_j}(\eta)$ such that $\dom\xi=\dom\eta$. If ($\iota$) holds, define
    $\xi(n)=F^{-1}_{\a_{j-1}}(\eta)(n)$ for all $n<\dom\eta$. If $\dom \eta=\o$, then 
    clearly $c_f(\xi)=c_g(\eta)$ by the induction hypothesis (specially (d)).
    Suppose that ($\iota\iota$) holds
    and let $m$ witness this. For all $n<\dom \xi$ let
    \begin{myItemize}
    \item If $n<m$, then $\xi(n)=F^{-1}_{\a_{j-1}}(\eta\rest m)(n)$.
    \item Suppose $n\ge m$. Let 
      \begin{myEnumerate}
      \item[$\cdot$] $\xi_1(n)=\xi_1(m-1)+1$,
      \item[$\cdot$] $\xi_2(n)=\xi_3(m-1)+\xi_4(m-1)$,
      \item[$\cdot$] $\xi_3(n)=\xi_2(m)+\rusi(\a_j)$,
      \item[$\cdot$] $\xi_4(n)=\g_{\eta\restl m}$,
      \item[$\cdot$] $\xi_5(n)=h_{\eta\restl m}^{-1}(\eta)(n)$.
      \end{myEnumerate}
    \end{myItemize}
    Next we should check that $\xi\in J_f$ and if $\dom \eta=\o$, also that $c_f(\xi)=c_g(\eta)$; 
    let us check items \eqref{J0} and \eqref{J6}, the rest are left to the reader.
    \begin{myItemize}
    \item[\eqref{J0}] By the induction hypothesis $\xi\rest m$ is increasing. Next,
      $\xi_1(m)=\xi_1(m-1)+1$, so $\xi(m-1)<_{\text{lex}}\xi(m)$. If $m\le n_1<n_2$,
      then $\xi_k(n_1)=\xi_{k}(n_2)$ for all $k\in\{1,2,3,4\}$ and $\xi_5$ is increasing.
    \item[\eqref{J6}] Suppose that $[i,j)=\xi_1^{-1}\{k\}$. Since $\xi_1\rest [m,\o)$ is constant,
      either  $j<m$, when we are done by the induction hypothesis, or $i=m$ and $j=\o$. In that case
      one verifies that $\eta\rest[m,\o)\in W=\ran h_{\eta\restl m}$ and then, imitating
      the corresponding argument in the first step, that 
      $$\xi_5\rest [m,\o)=h_{\eta\restl m}^{-1}(\eta\rest [m,\o))$$
      and hence in $\dom h_{\eta\restl m}=P^{\xi_2(m),\xi_3(m)}_{\xi_4(m)}$.
    \end{myItemize}   
    
    Suppose finally that ($\iota\iota\iota$) holds. Then $\dom\eta$ must be $\o$ since 
    otherwise the condition ($\iota\iota\iota$)
    is simply contradictory 
    (because $\eta\rest(\dom\eta-1+1)=\eta$ (except for the case $\dom\eta=0$, 
    but then condition ($\iota$) holds and we are done)).
    By (g) of the induction hypothesis, 
    we have $\ran\eta_1=\o$, because otherwise we had $\eta\in \ran F_{\a_{j-1}}$.
    Let $F^{-1}_{\a_j}(\eta)=\xi=\Cup_{n<\o}F^{-1}_{\a_{j-1}}(\eta\rest n)$. 

    Evidently $\xi\rest n$ is in $J_f$ for all $n<\o$, so $\xi\in J_f$ by the remark after
    \eqref{J2}.
    Let us check that $c_f(\xi)=c_g(\eta)$. 

    First of all $\xi$ cannot be in $J^{\a_{j-1}}_f$, since
    otherwise, by (d) and (i), 
    $$F_{\a_{j-1}}(\xi)=\Cup_{n<\o}F_{\a_{j-1}}(\xi\rest n)=\Cup _{n<\o}\eta\rest n=\eta$$ 
    were again in $\ran F_{\a_{j-1}}$. But $\xi\rest n$ is in $J^{\a_{j-1}}_f$, so by the definition
    of $J^{\a}_f$, $\a_{j-1}$ must be a limit ordinal, for otherwise also $\xi$ were in $J^{\a_{j-1}}_f$.
    Now by (b),
    $\a_{j-1}$ is a limit ordinal in $C$ and by (a), (e) and (f),
    $\ran F_{\a_{j-1}}=J^{\a_{j-1}}_g$ and $\dom F_{\a_{j-1}}=J^{\a_{j-1}}_f$. This implies 
    that $\ran\eta\not\subset \o\times \b^4$ for any $\b<\a_{j-1}$ 
    and by ($\circledast$) on page \pageref{circledast} we must have $\sup\ran\eta_5=\a_{j-1}$,
    so in particular $\a_{j-1}$ has cofinality $\o$. Therefore $c_g(\eta)=f(\a_{j-1})$.
    by \eqref{J7}. 
    Since $\a_{j-1}\in C$, we have
    $f(\a_{j-1})=g(\a_{j-1})$. Again by $(\circledast)$ and that $\dom F_{\a_{j-1}}=J^{\a_{j-1}}_f$ by (e), 
    we have $\sup\ran\xi_5=\a_{j-1}$ and $c_f(\xi)=f(\a_{j-1})=c_g(\eta)$, 
    thus the colors match. 

    Let us check whether all the conditions (a)-(i) are met. (a), (b), (c) are 
    common to the cases ($\iota$), ($\iota\iota$) 
    and ($\iota\iota\iota$) in the definition of $F_{\a_j}^{-1}$ and are easy to verify. 
    Let us sketch a proof for (d); the rest is left to the reader.
    \begin{myAlphanumerate}
    \item[(d)] We have already checked that the colors are preserved in the non-trivial cases.
      Let $\eta_1,\eta_2\in \ran F_{\a_{j}}$ and let us show that 
      $$\eta_1\subsetneq\eta_2\iff F^{-1}_{\a_j}(\eta_1)\subsetneq F^{-1}_{\a_j}(\eta_2).$$
      The case where both $\eta_1$ and $\eta_2$ satisfy $(\iota\iota)$ is the interesting one (implies all the others).
      
      So suppose $\eta_1,\eta_2\in (\iota\iota)$. Then there exist
      $m_1$ and $m_2$ as described in the statement of ($\iota\iota$).
      Let us show that $m_1=m_2$. We have $\eta_1\rest (m_1+1)=\eta_2\rest (m_1+1)$ 
      and $\eta_1\rest (m_1+1)\notin \ran F_{\a_{j-1}}$, 
      so $m_2\le m_1$. If $m_2\le m_1$, then $m_2<\dom\eta_1$, 
      since $m_1<\dom \eta_1$.  Thus if $m_2\le m_1$, then
      $\eta_1\rest (m_2+1)=\eta_2\rest (m_2+1)\notin \ran F_{\a_{j-1}}$, 
      which implies $m_2=m_1$.  According to the
      definition of $F^{-1}_{\a_j}(\eta_i)(k)$ for $k<\dom \eta_1$,
      $F^{-1}_{\a_j}(\eta_i)(k)$ depends only on $m_i$ and $\eta\rest m_i$ 
      for $i\in \{1,2\}$. Since $m_1=m_2$ and $\eta_1\rest m_1=\eta_2\rest m_2$, we have
      $F^{-1}_{\a_j}(\eta_1)(k)=F^{-1}_{\a_j}(\eta_2)(k)$ for all
      $k<\dom\eta_1$.

      Let us now assume that $\eta_1\not\subset \eta_2$. Then take the smallest $n\in \dom\eta_1\cap\dom\eta_2$ such that
      $\eta_1(n)\ne \eta_2(n)$. It is now easy to show that 
      $F^{-1}_{\a_j}(\eta_1)(n)\ne F^{-1}_{\a_j}(\eta_2)(n)$ by the construction.
    \end{myAlphanumerate}
    \noindent\textbf{Even successor step.} Namely the one where $j=\b+n$, $\b$ is limit and $n$ is even and $n>0$.
    But this case goes exactly as the odd successor step when it is not the successor of a limit, 
    except that we start with $\dom F_{\a_j}=J^{\a_j}_f$
    where $\a_j$ is big enough successor of an element of $C$ such that $J^{\a_j}_f$ contains $\ran F_{\a_{j-1}}$
    and define $\xi=F_{\a_j}(\eta)$. Instead of (e) we use (f) as the induction hypothesis. 
  
    \noindent\textbf{Limit step.}
    Assume that $j$ is a limit ordinal. Then let $\a_j=\Cup_{i<j}\a_i$ and $F_{\a_j}=\Cup_{i<j}F_{\a_i}$.
    Since $\a_i$ are successors of ordinals in $C$, $\a_j\in C$, so (b) is satisfied. 
    Since each $F_{\a_i}$ is an isomorphism, 
    also their union is, so (d) is satisfied. 
    Because conditions (e), (f) and (i) hold for $i<j$, the conditions (e) and (i)
    hold for $j$. (f) is satisfied because the premise is not true.
    (a) and (c) are clearly satisfied. Also (g) and (h) are satisfied by Claim~1 since now $\dom F_{\a_j}=J^{\a_j}_f$
    and $\ran F_{\a_j}=J^{\a_j}_g$ (this is because (a), (e) and (f) hold for $i<j$).

    \noindent\textbf{Finally} $F=\Cup_{i<\k}F_{\a_i}$ is an isomorphism between $J_f$ and $J_g$.
  \end{proofVOf}
\end{proofV}

\begin{Def}\label{def:trees}
  Let $K(\k^+,\o+\o+2)$ be the class of those $\k^+,(\o+\o+2)$-trees which
  have the following properties:
  \begin{myItemize}
  \item every node on level $<\o+\o$ has infinitely many immediate successors,
  \item for every node $x$ there exists $y$ on level $\o+\o$ (a leaf) such that $x<y$.
  \end{myItemize}
\end{Def}

\begin{Thm}[$V=L$, $\k=\l^+$, $\l^\o=\l$]\label{thm:Complete2}
  The isomorphism relation on $K(\k^+,\o+\o+2)$ is $\Sii$-complete.
\end{Thm}
\begin{proof}
  By Theorem \ref{thm:Complete1} every $\Sii$-equivalence relation can be reduced 
  to the equivalence relation on $\l^\k$ modulo the $\o$-non-stationary ideal.
  By Lemma \ref{lem:StoJS} this equivalence relation can embedded to the isomorphism
  relation on $\CT^\o_*$ (Definition \ref{def:CT}). So it remains to show that 
  the isomorphism relation on $\CT^\o_*$ can be embedded into the isomorphism relation on
  $K(\k^+,\o+\o+2)$.
  
  Let $(t_i)_{i<\l}$ be a sequence of non-isomorphic $\k^+,(\o+2)$-trees 
  such that every element has infinitely many immediate successors and
  every element has a successor at level $\o$.
  These can be obtained for example by Lemma 4.89 in 
  \cite{FHK} which is essentially the same as Lemma~\ref{lem:StoJS}
  above but with the number of colors $\l$ replaced by $2$ (a branch either has a leaf or not).
  Let $(t,c)\in \CT^\o_*$. 
  Let $F(t,c)$ be the tree obtained from $t$ as follows: if $x\in t$ has height $\o$, then
  by definition it has a color $\a<\l$, so replace the element $x$ by the tree $t_\a$.
  
  Clearly if colored trees $(t,c)$ and $(t',c')$ are isomorphic, then so are the trees
  $F(t,c)$ and $F(t',c')$. On the other hand if $g\colon F(t,c)\to F(t',c')$ is an isomorphism,
  then $g\rest t$ must be an isomorphism onto $t'$. Moreover it preserves colors, because
  if $c(x)\ne c'(g(x))$ for some $x$ of height $\o$ in $t$, then $g\rest \{y\in F(t,c)\mid \forall z<x(y>z)\}$ must 
  be an isomorphism onto $\{y\in F(t',c')\mid y>x\}$ which means that there is an isomorphism 
  between $t_{c(x)}$ and $t_{c'(g(x))}$ which is a contradiction.
  
  It remains to show that $(t,c)\mapsto F(c,t)$ is continuous.
  But this follows from the fact that $|t_\a|\le |t_\l|\le |\l^{<\o}|=\l<\k$.
  %
\end{proof}

\begin{Def}\label{def:Too}
  Let $A=\o^{\o+\o}$ and for $\a<\o+\o$, let $E_{\a}$ be the equivalence relation on $A$ such that
  let $(\eta,\xi)\in E_{\a}\iff \eta\rest\a=\xi\rest\a$. 
  Let $\A$ be a model of the vocabulary $(E_n)_{n<\o+\o}$ 
  with domain $A$ and all these equivalence relations interpreted as explained above. Then
  denote the complete theory of $\A$, $T(\A)$ by $T_{\o+\o}$. 
\end{Def}

\begin{Fact}
  $T_{\o+\o}$ is stable and has no DOP nor OTOP. 
\end{Fact}

\begin{Fact}[\cite{Hyt1}]\label{fact}
  Suppose that $\k=\l^+$, $\l$ is regular and $\k\in I[\k]$. (Here $I[\k]$ is the ideal of the approachable sets
  introduced by S. Shelah, for reference see for example~\cite{HHR}. For our purposes this assumption is weak, because
  $\k\in I[\k]$ holds for all $\k$ with $\k=\l^+$ and $\l^{<\l}=\l$.)
  If $t$ and $t'$ are elements of $K(\k^+,\o+\o+2)$
  and player $\PlTwo$ has a winning strategy
  in the Ehrenfeucht-Fra\"iss\'e game $\EF^\k_{\l\cdot (\o+\o+1)}(t,t')$, then $t\cong t'$.
  An earlier, less general version of this theorem along with a proof can be found in~\cite{HS1}.
  Another version can also be found in~\cite{HT}.
\end{Fact}

\begin{Lemma}\label{lemma:trTbired}
  The relations $\ISO(K(\k^+,\o+\o+2))$ and $\ISO(\k,T_{\o+\o})$ are continuously bireducible
  to each other
  and Fact \ref{fact} holds with $K(\k^+,\o+\o+2)$ replaced by $\M(T_{\o+\o})$.
\end{Lemma}
\begin{proof}
  Suppose $t\in K(\k^+,\o+\o+2)$. Let $M$ be the set of all leaves of $t$, i.e. elements on level
  $\o+\o$. For all $\a<\o+\o$ and $a,b\in M$, set $(a,b)\in E_\a^M$ if and only if there exists $x\in t$
  on level $\a$ with $x<a$ and $x<b$. Then $t\mapsto M$ defines a map 
  $$F\colon K(\k^+,\o+\o+2)\to \M(T_{\o+\o}).$$
  Clearly if trees $t$ and $t'$
  are isomorphic, then so are $F(t)$ and $F(t')$. Suppose $F(t)$ and $F(t')$ are isomorphic
  via the isomorphism~$g\colon F(t)\to F(t')$.
  This induces a bijection $f$ from the leaves of $t$ to the leaves of $t'$ which preserves the 
  the pairwise splitting levels of the branches. If $a$ and $b$ are leaves, denote by $s(a,b)$
  the smallest $\a$ such that there is no $x$ on level $\a$ with $x<a$ and $x<b$.
  Then the above can be written $s(f(a),f(b))=s(a,b)$ $(*)$.
  Now we can extend $f$
  to the whole tree using this information as follows. Let $x\in t$ and let $a$ be any leaf above $x$.
  Such exists by the definition of $K(\k^+,\o+\o+2)$. Let $f(x)$ be the unique element in $t'$ below
  $f(a)$ which is on the same level as $x$ is in $t$. Then $f$ is well defined: if $a$ and $b$ are
  two different leaves above $x$, then $s(a,b)>\a$, where $\a$ is the level of $x$, so by $(*)$ we have
  $s(f(a),f(b))>\a$ and so $f(x)$ is independent on which branch is used.




  Deciding whether an element $x\in t$ is a leaf or not requires only countable information
  and that is why the described reduction is continuous.  

  Let us now define a continuous reduction $G$ from $T_{\o+\o}$ to $K(\k^+,\o+\o+2)$.
  Suppose $M$ is a model of $T_{\o+\o}$
  Let $E_{\o+\o}^M$ the identity relation on $\dom M$: $(a,b)\in E_{\o+\o}$ if and only
  $a=b$. Let
  $G(M)=\Cup_{\a\le \o+\o}\dom M/E_\a$ and for $x,y\in G(M)$
  let $x<y$, if $y\subset x$. 
  Clearly $G(M)\in K(\k^+,\o+\o+2)$ and in fact
  $G(F(t))=t$ for all $t\in K(\k^+,\o+\o+2)$ which implies the rest.

  Using this construction it is easy to see that if player $\PlTwo$ has a winning strategy in
  $\EF^\k_{\l\cdot (\o+\o+1)}(M,M')$ for $M,M'\in \M(T_{\o+\o})$, then
  she has a winning strategy also in the game $\EF^\k_{\l\cdot(\o+\o+1)}(G(M),G(M'))$. So
  if this implies that $G(M)\cong G(M')$ (and it does under the assumptions of Fact~\ref{fact}),
  then it also implies~$M\cong M'$.
\end{proof}

\begin{Cor}[$V=L$, $\k=\l^+$, $\l^\o=\l$]\label{thm:Stable1}
  $\ISO(\k,T_{\o+\o})$ is $\Sii$-complete.
\end{Cor}
\begin{proof}
  By the lemma above it is sufficient to look at the trees in the class $K(\k^+,\o+\o)$
  and the result follows from Theorem~\ref{thm:Complete2}.
\end{proof}

A similar result to Theorem \ref{thm:Complete4} for computable reductions has been observed in~\cite{FFKMM}:

\begin{Thm}[ZFC, $\k^{<\k}=\k>\o$]\label{thm:Complete4}
  Let $\dlo$ be the theory of dense linear orderings without end points.
  Then $\ISO(\k,\dlo)$ is $S_\k$-complete.
\end{Thm}
\begin{proof}  
  It was proved in \cite{FrSt} that the isomorphism relation
  on all countable binary structures is reducible to countable linear orderings.
  The same proof works for $\k>\o$. 
  Then we embed all linear orders into dense linear orders
  by replacing each point by the ordering $\eta+\Q+\eta$, where $\eta$ is the saturated $\dlo$ of
  size $\k$ and $\Q$ is the countable saturated~$\dlo$.
\end{proof}

\begin{Thm}[\cite{HK}]\label{thm:DLONotBorelSt}
  Suppose $\k^+=2^{\k}$ and $\k^{<\k}=\k$. Then there exists a $<\k$-closed, $\k^+$-c.c. forcing
  $\P$ which forces that $\ISO(\k,\dlo)$ is not $\Borel^*$ and at 
  the same time $\Dii\subsetneq \Borel^*$. \qed
\end{Thm}

\begin{Cor}\label{cor:StableNotCom}
  Suppose $\k^+=2^\k$, $\k^{<\k}=\k=\l^+$ and $\l^{<\l}=\l$.
  Then there is a $<\k$-closed $\k^+$-c.c. forcing which forces that
  $\ISO(\k,T_{\o+\o})$
  is not $S_\k$-complete, in particular not $\Sii$-complete. 
\end{Cor}
\begin{proof}
  By Theorem \ref{thm:DLONotBorelSt} there is a $<\k$-closed $\k^+$-c.c. forcing which forces
  that $\ISO(\k,\dlo)$ is not $\Borel^*$. This forcing preserves cardinals and preserves the
  fact $\l^{<\l}=\l$ and so also that $I[\k]$ is improper.
  Thus in the forced model the isomorphism of $T_{\o+\o}$
  can be characterized by the $\EF$-game of length $\l\cdot(\o+\o+1)$ by Fact~\ref{fact}.
  This implies that $\ISO(\k,T_{\o+\o})$ is $\Borel^*$ 
  which can be seen in the same way as the $\Leftarrow$-part of
  Theorem 4.68 in~\cite{FHK}. Therefore $\ISO(\k,\dlo)$ cannot be reduced to it because 
  a non-$\Borel^*$ relation cannot be Borel-reduced to a $\Borel^*$ equivalence relation.
  To see this let $f\colon \k^\k\to\k^\k$ be a Borel map and let $B\subset \k^\k$ be a 
  Borel* set. It is sufficient to show that the inverse image of $B$ is also Borel*.
  First note that in the definition of the Borel* sets the basic open sets can 
  be replaced by Borel sets the definition remaining equivalent; let us call such 
  Borel*-codes \emph{extended Borel*-codes}. So then take the
  Borel*-code $(t,h)$ of $B$ and let $(t,k)$ be an extended Borel*-code with the same tree
  $t$ and $k(b)=f^{-1}h(b)$ for all leaves $b$ and otherwise $k$ gets the same values as
  $h$. Now it is easy to see that $(t,k)$ is an extended Borel*-code for~$f^{-1}B$.
\end{proof}

\bibliography{ref}{}

\newcommand{\etalchar}[1]{$^{#1}$}
\providecommand{\bysame}{\leavevmode\hbox to3em{\hrulefill}\thinspace}
\providecommand{\MR}{\relax\ifhmode\unskip\space\fi MR }
\providecommand{\MRhref}[2]{%
  \href{http://www.ams.org/mathscinet-getitem?mr=#1}{#2}
}
\providecommand{\href}[2]{#2}
\begin{thebibliography}{FFK{\etalchar{+}}12}

\bibitem[Bla81]{Bl}
D.~Blackwell, \emph{Borel sets via games.}, Ann. Probab. \textbf{9} (1981),
  no.~2, 321--322.

\bibitem[FFK{\etalchar{+}}12]{FFKMM}
E.~Fokina, S.~D. Friedman, J.~Knight, R.~Miller, and A.~Montalb\'an,
  \emph{Classes of structures with universe a subset of {$\o_1$}}, To appear
  (2012).

\bibitem[FHK11]{FHK}
S.~D. Friedman, T.~Hyttinen, and V.~Kulikov, \emph{Generalized descriptive set
  theory and classification theory}, Centre de Recerca M\`athematica, CRM,
  Barcelona, preprint \textbf{999} (2011).

\bibitem[FS89]{FrSt}
H.~Friedman and L.~Stanley, \emph{A borel reducibility theory for classes of
  countable structures}, Journal of Symbolic Logic \textbf{54} (1989), no.~3,
  894--914.

\bibitem[Hal96]{Ha}
A.~Halko, \emph{Negligible subsets of the generalized {Baire} space
  {$\o_1^{\o_1}$}}, Ann. Acad. Sci. Ser. Diss. Math. \textbf{108} (1996).

\bibitem[HHR99]{HHR}
T.~Huuskonen, T.~Hyttinen, and M.~Rautila, \emph{On the {$\k$}-cub game on
  {$\l$} and {$I[\l]$}}, Arch. Math. Logic \textbf{38} (1999), 549--557.

\bibitem[HK12]{HK}
T.~Hyttinen and V.~Kulikov, \emph{On {$\Sigma^1_1$}-complete equivalence
  relations on the generalized {B}aire space}, Submitted (2012).

\bibitem[HS94]{HS1}
Tapani Hyttinen and Saharon Shelah, \emph{Constructing strongly equivalent
  nonisomorphic models for unstable theories, part {A}}, Journal of Symbolic
  Logic \textbf{59} (1994), 984--996.

\bibitem[HT91]{HT}
T.~Hyttinen and H.~Tuuri, \emph{Constructing strongly equivalent nonisomorphic
  models}, Annals of Pure and Applied Logic \textbf{52} (1991), no.~3,
  203--248.

\bibitem[Hyt01]{Hyt1}
T.~Hyttinen, \emph{Club-guessing and non-structure of trees}, Fund. Math.
  \textbf{168} (2001), 237--249.

\bibitem[Jec03]{Jech}
T.~Jech, \emph{Set theory}, Springer-Verlag Berlin Heidelberg New York, 2003.

\bibitem[MV93]{MV}
A.~Mekler and J.~V\"a\"an\"anen, \emph{Trees and {$\Pii$}-subsets of
  {${}^{\o_1}\o_1$}}, The Journal of Symbolic Logic \textbf{58} (1993), no.~3,
  1052--1070.

\bibitem[She00]{Sh}
S.~Shelah, \emph{Classification theory, revised edition}, North Holland
  Publishing Company, 2000.

\end{thebibliography}
\bibliographystyle{amsalpha}

\end{document}